\documentclass[leqno,12pt]{article}
\usepackage{amsfonts}
\usepackage{amsmath}
\usepackage{amssymb}
\usepackage{cite}
\newtheorem{theorem}{Theorem}
\newtheorem{prop}{Property}
\newtheorem{rem}{Remark}
\newtheorem{dfn}{Definition}

\def\email#1{{\tt #1}}

\title{The simplified topological $\varepsilon$--algorithms for accelerating sequences in a vector space}

\author{Claude Brezinski\thanks{Laboratoire Paul Painlev\'e, UMR CNRS 8524, UFR
de Math\'ematiques, Universit\'e des Sciences
et Technologies de Lille,
59655--Villeneuve d'Ascq cedex,
France (\email{Claude.Brezinski@univ-lille1.fr})}
\and Michela Redivo--Zaglia\thanks{Universit\`a degli Studi di Padova,
Dipartimento di Matematica,
Via Trieste 63, 35121--Padova,
Italy (\email{Michela.RedivoZaglia@unipd.it})}
}

\begin{document}
\maketitle

\begin{abstract}
When a sequence of numbers is slowly converging, it can be transformed into a new sequence which, under some assumptions, could converge faster to the same limit. One of the most well--known sequence transformation is Shanks transformation which can be recursively implemented by the $\varepsilon$--algorithm of Wynn. This transformation and this algorithm have been extended (in two different ways) to sequence of elements of a topological vector space $E$. In this paper, we  present new algorithms of implementing these topological Shanks transformations. They no longer require the manipulation of elements of the algebraic dual space $E^*$ of $E$, nor using the duality product into the rules of the algorithms, they need the storage of less elements of $E$, and the stability is improved. They also allow us to prove convergence and acceleration results for some types of sequences.
Various applications involving sequences of vectors or matrices show the interest of the new algorithms.
\end{abstract}



\pagestyle{myheadings}
\thispagestyle{plain}
\markboth{\sc The simplified topological $\varepsilon$--algorithms}{\sc C. Brezinski and M. Redivo--Zaglia}

\section{Presentation}

Let $(\mathbf S_n)$ be a sequence of elements of a vector space $E$ on a field $\mathbb K$ ($\mathbb R$ or $\mathbb C$). If the sequence $(\mathbf S_n)$ is slowly converging to its limit $\mathbf S$ when $n$ tends to infinity, it can be transformed, by a {\it sequence transformation}, into a new sequence converging, under some assumptions, faster to the same limit. The construction of most sequence transformations starts from the notion of {\it kernel} which is the set of sequences which are transformed into a constant sequence all members of which are equal to $\mathbf S$
(see \cite{cbmrz}).

\vskip 2mm

In Section \ref{shat}, the scalar Shanks transformation for transforming a sequence of numbers \cite{shanks},
and its implementation by the scalar $\varepsilon$--algorithm \cite{wynn} are remembered. The two versions
of the
{\it topological Shanks transformation} proposed by Brezinski \cite{etopo} for sequences of elements of
a vector space are reminded in Section \ref{tsst}, as well as their original implementation by two kinds
of a {\it topological} $\varepsilon${\it--algorithm} in Section \ref{tea}. They both need to perform operations
involving elements of the algebraic dual space $E^*$ of $E$.

The topological Shanks transformations and the topological $\varepsilon$--algorithms received much
attention in the literature, so that all contributions cannot be quoted here. They have many applications
in the solution of systems of linear and nonlinear equations, in the computation of eigenelements,
in Pad\'{e}--type approximation, in matrix functions and matrix equations,
in Lanczos method, etc. See \cite{birk,jrs,sal1,cbmrz,jbi,jsnuma,anm,jmt,smith2,smith22,tan} and their references,
in particular, for the older ones.

Two new algorithms for implementing these topological Shanks transformations, the first and the second
{\it simplified topological} $\varepsilon${\it --algorithm}, are introduced in Section \ref{tsea}.
They have several important advantages: elements of the dual vector space $E^*$ have no longer to be
used in their recursive rules and are replaced by quantities computed by the scalar $\varepsilon$--algorithm,
less vectors have to be stored than for implementing the transformations by the topological
$\varepsilon$--algorithms, a very important issue in applications, and the numerical stability of
the algorithms is improved. Moreover, these new algorithms allow us to prove convergence
and acceleration properties of the transformations that could hardly have been obtained from the
original topological $\varepsilon$--algorithms (Section \ref{conacc}).
The implementation of these algorithms is discussed in Section \ref{impl}. Some applications involving
sequences of vectors and matrices are presented in Section \ref{appli}.

\vskip 2mm

In this paper, characters in bold correspond to elements of a general vector space $E$ or of its
algebraic dual $E^*$, while regular ones designate real or complex numbers.

\section{Shanks transformation and the $\varepsilon$--algorithm}
\label{shat}

For sequences $(S_n)$ of real or complex numbers, an important sequence transformation is Shanks
transformation \cite{shanks} whose kernel consists of sequences satisfying the homogeneous linear
difference equation of order $k$
\begin{equation}
a_0(S_n-S)+\cdots+a_k(S_{n+k}-S)=0,\quad n=0,1,\ldots, \label{aa1}
\end{equation}
where the $a_i$'s are arbitrary constants independent of $n$ such that $a_0a_k \neq 0$, and
$a_0+\cdots+a_k \neq 0$, and where
the unknown $S$ is the limit of $(S_n)$ if it converges, or is called its antilimit otherwise.
The exact expression of such sequences is given in \cite{cbmc}.
Assuming that (\ref{aa1}) holds for all $n$, the problem is to compute $S$.
We have, for all $n$,
\begin{equation}
a_0 \Delta S_n+\cdots+a_k \Delta S_{n+k}=0, \label{rr1}
\end{equation}
where the forward difference operator $\Delta$ is defined by $\Delta^{i+1}
S_n=\Delta^i S_{n+1}-\Delta^i S_n$, with
$\Delta^0 S_n=S_n$.
Writing (\ref{rr1}) for the indexes $n,\ldots,n+k-1$ produces $k$ homogeneous linear equations
in the $k+1$ unknowns $a_0,\ldots,a_k$. Adding the condition $a_0+\cdots+a_k=1$ (which does not
restrict the generality), and solving the system obtained, gives the unknowns, and then, from
(\ref{aa1}), we obtain, for all $n$,
$$S=a_0S_n+\cdots+a_kS_{n+k}.$$

Now, if $(S_n)$ does not satisfy (\ref{aa1}), we can still write the system giving the unknown
coefficients and compute the linear combination $a_0S_n+\cdots+a_kS_{n+k}$. However, these
coefficients and the linear combination will now depend on $n$ and $k$, and they will be
respectively denoted $a_i^{(n,k)}$ and $e_k(S_n)$, and we have
$$e_k(S_n)=a_0^{(n,k)}S_n+\cdots+a_k^{(n,k)}S_{n+k}, \quad k,n=0,1,\ldots$$
where the $a_i^{(n,k)}$ are solution of the same system as before, that is
\begin{equation}
\label{scTEA}
\left\{
\begin{array}{ccccccc}
a_0^{(n,k)} &+& \cdots &+& a_k^{(n,k)} &=& 1\\
a_0^{(n,k)} \Delta S_n&+&\cdots &+& a_k^{(n,k)} \Delta S_{n+k}&=& 0\\
\vdots &&&& \vdots &&\\
a_0^{(n,k)} \Delta S_{n+k-1}&+&\cdots &+& a_k^{(n,k)} \Delta S_{n+2k-1}&=& 0.
\end{array}
\right.
\end{equation}

Thus, the sequence $(S_n)$ has been transformed into the set of sequences $\{(e_k(S_n))\}$.
The transformation $(S_n) \longmapsto (e_k(S_n))$ is called {\it Shanks transformation}. It is a generalization of
the well--known Aitken's $\Delta^2$ process which is recovered for $k=1$.

Cramers' rule allows us to write $e_k(S_n)$ as the following ratio of determinants
\begin{equation*}
e_k(S_n)=\frac{\left|
\begin{array}{ccc}
S_n & \cdots & S_{n+k}\\
\Delta S_n & \cdots & \Delta S_{n+k}\\
\vdots && \vdots\\
\Delta S_{n+k-1} & \cdots & \Delta S_{n+2k-1}
\end{array}
\right|}
{\left|
\begin{array}{ccc}
1 & \cdots & 1\\
\Delta S_n & \cdots & \Delta S_{n+k}\\
\vdots && \vdots\\
\Delta S_{n+k-1} & \cdots & \Delta S_{n+2k-1}
\end{array}
\right|}, \quad k,n=0,1,\ldots
%\label{aa2}
\end{equation*}

By construction, we have the

\begin{theorem} {\em \cite{cbmrz}}
\label{cns}
For all $n$, $e_k(S_n)=S$ if and only if $\exists a_0,\ldots,a_k$, with $a_0a_k \neq 0$ and $a_0+\cdots+a_k \neq 0$, such that, for all $n$,
$$a_0(S_n-S)+\cdots+a_k(S_{n+k}-S)=0,$$
that is, in other words, if and only if $(S_n)$ belongs to the kernel of Shanks transformation $(S_n) \longmapsto (e_k(S_n))$ for $k$ fixed.
\end{theorem}

\vskip 2mm

Shanks transformation can be recursively implemented by the
scalar $\varepsilon$--algorithm of Wynn \cite{wynn}, whose rules are
\begin{equation}
\label{esca}
\left\{
\begin{array}{lcl}
\varepsilon_{-1}^{(n)}&=&0,\qquad n=0,1,\ldots,\\
\varepsilon_{0}^{(n)}&=&S_n,\qquad n=0,1,\ldots,\\
\varepsilon_{k+1}^{(n)}&=&\varepsilon_{k-1}^{(n+1)}+
(\varepsilon_{k}^{(n+1)}-\varepsilon_{k}^{(n)})^{-1}, \qquad k,n=0,1,\ldots,
\end{array}
\right.
\end{equation}
and it holds
\begin{prop}
\label{esc}
For all $k$ and $n$, $\varepsilon_{2k}^{(n)}=e_k(S_n)$ and $\varepsilon_{2k+1}^{(n)}=1/e_k(\Delta S_n)$.
\end{prop}

These elements are usually displayed in a two dimensional array, the so--called $\varepsilon$--array
(see Table \ref{tabeps} for a triangular part of it), where the quantities with an odd lower index are
only intermediate computations. The rule (\ref{esca}) relates numbers located at the four vertices of
a rhombus in Table \ref{tabeps}.

\vskip 2mm

\begin{table}[htb]
\[
\begin{array}{cccccc}
\varepsilon_{-1}^{(0)} = 0 & & & & &  \\
& \varepsilon_{0}^{(0)} = S_0 & & & &  \\
\varepsilon_{-1}^{(1)} = 0 & & \varepsilon_{1}^{(0)} & & &  \\
& \varepsilon_{0}^{(1)} = S_1 & & \varepsilon_{2}^{(0)} & &  \\
\varepsilon_{-1}^{(2)} = 0 & & \varepsilon_{1}^{(1)} &
& \varepsilon_{3}^{(0)} &  \\
& \varepsilon_{0}^{(2)} = S_2 & & \varepsilon_{2}^{(1)} & & \varepsilon_{4}^{(0)} \\
\varepsilon_{-1}^{(3)} = 0 & & \varepsilon_{1}^{(2)} & &
\varepsilon_{3}^{(1)} &  \\
& \varepsilon_{0}^{(3)} = S_3 & &
\varepsilon_{2}^{(2)} & &   \\
\varepsilon_{-1}^{(3)} = 0 &  & \varepsilon_{1}^{(3)} & &
 &  \\
 & \varepsilon_{0}^{(4)} = S_4 &  &
 & &   \\
\varepsilon_{-1}^{(4)} = 0 &  &  &  &
&  \\
\end{array}
\]
\caption{A triangular part of the $\varepsilon$--array\label{tabeps}.}
\end{table}

In passing, let us give two relations which seem not to have been noticed earlier.
Their proofs are straightforward by induction.

\begin{prop}
\label{prop1}
For all $k$ and $n$,
$$
\varepsilon_{2k}^{(n)}=S_{n+k}+\sum_{i=1}^k 1/\Delta \varepsilon_{2i-1}^{(n+k-i)}, \qquad
\varepsilon_{2k+1}^{(n)}=\sum_{i=0}^k 1/\Delta \varepsilon_{2i}^{(n+k-i)}.$$
\end{prop}

\section{The topological Shanks transformations}
\label{tsst}

Let us now consider the case of a sequence of elements $\mathbf S_n$ of a vector space $E$.
In 1962, Wynn generalized the scalar $\varepsilon$--algorithm to sequences of vectors $(\mathbf S_n)$
by defining, in the rule of the scalar $\varepsilon$--algorithm,
the inverse of a vector $\mathbf u \in E=\mathbb R^m$ by
$\mathbf u^{-1}=\mathbf u/(\mathbf u,\mathbf u)$, where $(\cdot,\cdot)$ is the
usual inner product \cite{evec}.
He thus obtained the {\it vector $\varepsilon$--algorithm} ({\sc vea})
which had no algebraic foundation. However, it was later proved by McLeod \cite{mcl}
that the kernel of this vector $\varepsilon$--algorithm is the set of vector
sequences satisfying a relation of the same form as (\ref{aa1}) with $a_i \in \mathbb R$.
The proof was quite technical and it involved Clifford algebra. It was later showed that
the vectors $\boldsymbol\varepsilon_{2k}^{(n)} \in \mathbb R^m$ it computes are given by
ratios of determinants of dimension $2k+1$ instead of $k+1$ as in the scalar case \cite{prgm},
or by designants of dimension $k+1$, objects that generalize determinants in a non--commutative
algebra \cite{5sala}.

\vskip 2mm

To remedy the lack of an algebraic theory for the vector $\varepsilon$--algorithm similar to that existing for the scalar one, Brezinski, following the approach of Shanks, obtained two versions of the so--called {\it topological Shanks transformation} \cite{etopo}
which can, more generally, be applied to a sequence of elements of a topological vector space $E$ (thus the name) on $\mathbb K$ ($\mathbb R$ or $\mathbb C$),
since, for being able to deal with convergence results, $E$ has to be a topological vector space.
Starting from a relation of the same form as (\ref{aa1}), that is
\begin{equation}
\label{eqd2}
a_0(\mathbf S_n-\mathbf S)+\cdots+a_k(\mathbf S_{n+k}-\mathbf S)={\mathbf 0} \in E,\quad n=0,1,\ldots,
\end{equation}
where $\mathbf S_n \in E$ and $a_i \in \mathbb K$ with $a_0a_k \neq 0$ and $a_0+\cdots+a_k \neq 0$, we write again
$$a_0 \Delta \mathbf S_n+\cdots+a_k \Delta \mathbf S_{n+k}=\mathbf 0$$
for the indexes $n,\ldots,n+k-1$. However, these relations do not allow to compute the numbers
$a_i$ since the $\mathbf S_n$ are elements of $E$
(in many practical applications they are vectors of $\mathbb R^m$ or square matrices of
dimension $m$). For that purpose, let $\mathbf y \in E^*$ (the algebraic dual space of $E$,
that is the vector space of linear functionals on $E$), and take the duality product of
these relations with it. We obtain
$$a_0 <\mathbf y,\Delta \mathbf S_i>+\cdots+a_k <\mathbf y,\Delta \mathbf S_{i+k}>=\mathbf 0,
\qquad i=n,\ldots,n+k.$$
As in the scalar case, if $(\mathbf S_n)$ does not satisfy (\ref{eqd2}), the preceding
relations, together with the additional condition that the coefficients $a_i$ sum up to 1,
can still be written. We thus obtain a system of $k+1$ equations in $k+1$ unknowns,
denoted $a_i^{(n,k)}$, similar to (\ref{scTEA}),

\begin{equation}
\label{sTEA}
\left\{
\begin{array}{ccccccc}
a_0^{(n,k)} &+& \cdots &+& a_k^{(n,k)} &=& 1\\
a_0^{(n,k)} <\mathbf y,\Delta \mathbf S_n>&+&\cdots &+& a_k^{(n,k)} <\mathbf y,\Delta \mathbf S_{n+k}>&=& 0\\
\vdots &&&& \vdots &&\\
a_0^{(n,k)} <\mathbf y,\Delta \mathbf S_{n+k-1}>&+&\cdots &+& a_k^{(n,k)} <\mathbf y,\Delta \mathbf S_{n+2k-1}>&=& 0,
\end{array}
\right.
\end{equation}
and we define the first topological Shanks transformation by
\begin{equation*}
%\label{trans}
\mathbf {\widehat{e}}_k(\mathbf S_n)=a_0^{(n,k)}\mathbf S_n+\cdots+a_k^{(n,k)}\mathbf S_{n+k}, \quad n,k=0,1,\ldots
\end{equation*}

Thus, as in the scalar case,  we have
$$\mathbf {\widehat{e}}_k(\mathbf S_n)=\frac{\left|
\begin{array}{ccc}
\mathbf S_n & \cdots & \mathbf S_{n+k}\\
<\mathbf y,\Delta \mathbf S_n> & \cdots & <\mathbf y,\Delta \mathbf S_{n+k}>\\
\vdots && \vdots\\
<\mathbf y,\Delta \mathbf S_{n+k-1}> & \cdots & <\mathbf y,\Delta \mathbf S_{n+2k-1}>
\end{array}
\right|}
{\left|
\begin{array}{ccc}
1 & \cdots & 1\\
<\mathbf y,\Delta \mathbf S_n> & \cdots & <\mathbf y,\Delta \mathbf S_{n+k}>\\
\vdots && \vdots\\
<\mathbf y,\Delta \mathbf S_{n+k-1}> & \cdots & <\mathbf y,\Delta \mathbf S_{n+2k-1}>
\end{array}
\right|}, \quad k,n=0,1,\ldots,
$$
where the determinant in the numerator denotes the element of $E$ obtained by developing it with respect to its first row by the classical rule for expanding a determinant.

\vskip 2mm

Similarly, the second topological Shanks transformation is defined by
$$
\mathbf{ {\widetilde e}}_k(\mathbf S_n)=a_0^{(n,k)}\mathbf S_{n+k}+\cdots+a_k^{(n,k)}\mathbf S_{n+2k}, \quad n,k=0,1,\ldots,
$$
where the $a_i^{(n,k)}$'s are again solution of the system (\ref{sTEA}), and, thus, are the same as the coefficients in the first topological Shanks transformation. These $\mathbf{ {\widetilde e}}_k(\mathbf S_n)$ are given by the same ratio of determinants as above after replacing the first row in the numerator by $\mathbf S_{n+k},\ldots,\mathbf S_{n+2k}$.

\vskip 2mm

The following fundamental property shows the connection between the scalar and the topological Shanks transformations

\begin{prop}
\label{prp}
Setting $S_n=<\mathbf y, \mathbf S_n>$, the systems (\ref{scTEA}) and (\ref{sTEA}) giving the coefficients $a_i^{(n,k)}$ of the scalar, the first and the second topological Shanks transformations are the same.
\end{prop}

We will see below that this property is a quite fundamental one, having important consequences for the convergence and the acceleration properties of the topological Shanks transformations.

\vskip 2mm

For the first and the second topological Shanks transformation, we have the

\begin{theorem}{\em \cite{etopo}}
For all $n$, $\mathbf {\widehat{e}}_k(\mathbf S_n)=\mathbf S$ and $\mathbf{ \widetilde e_k}(\mathbf S_n)=\mathbf S$ if $\exists a_0,\ldots,a_k$, with $a_0a_k \neq 0$ and $a_0+\cdots+a_k \neq 0$, such that, for all $n$,
$$a_0(\mathbf S_n-\mathbf S)+\cdots+a_k(\mathbf S_{n+k}-\mathbf S)=\mathbf 0.$$
\end{theorem}

\begin{rem}
Contrarily to Theorem \ref{cns}, the condition of this theorem is only sufficient. If the condition is satisfied, we also have,
for all $n$,
$$a_0<\mathbf y, \mathbf S_n-\mathbf S>+\cdots+a_k<\mathbf y,\mathbf S_{n+k}-\mathbf S>= 0.$$
\end{rem}
\indent Let us mention that other extrapolation methods based on a kernel of the same form, such as the {\sc rre},  the {\sc mpe}, the {\sc mmpe}, etc, have been defined (see, for example \cite{cbmrz}). They were recently studied in \cite{jsnuma}, including the first topological Shanks transformation.

The choice of $\mathbf y$ in the vector and matrix cases will be discussed in Section \ref{impl}.

\section{The topological $\boldsymbol\varepsilon$--algorithms}
\label{tea}

Let us now discuss how to compute recursively the elements $\mathbf {\widehat{e}}_k(\mathbf S_n) \in E$ and $\mathbf{\widetilde e}_k(\mathbf S_n) \in E$.

\vskip 2mm

The elements $\mathbf {\widehat{e}}_k(\mathbf S_n) \in E$ can be recursively computed by the
{\it first topological} $\varepsilon$--{\it algorithm} ({\sc tea{\small 1}}) \cite{etopo} whose rules are
\begin{equation}
\label{etop}
\left\{
\begin{array}{lcl}
\boldsymbol{\widehat{\varepsilon}}_{-1}^{(n)} \!\!&\!=\!&\!\! \mathbf 0 \in E^*, \qquad n=0,1,\ldots,\\
\boldsymbol{\widehat{\varepsilon}}_{0}^{(n)} \!\!&\!=\!&\!\! \mathbf S_n \in E, \qquad n=0,1,\ldots,\\
\boldsymbol{\widehat{\varepsilon}}_{2k+1}^{(n)} \!\!&\!=\!&\!\!  \boldsymbol{\widehat{\varepsilon}}_{2k-1}^{(n+1)}+
\displaystyle \frac{
{\mathbf{y}}}{<{\mathbf{y}},{\boldsymbol{\widehat{\varepsilon}}}^{(n+1)}_{2k}
- {\boldsymbol{\widehat{\varepsilon}}}^{(n)}_{2k}>}
\in E^*, \; \; k,n=0,1,\ldots,\\
\boldsymbol{\widehat{\varepsilon}}_{2k+2}^{(n)} \!\!&\!=\!&\!\!  \boldsymbol{\widehat{\varepsilon}}_{2k}^{(n+1)}+
\displaystyle \frac{
{\boldsymbol{\widehat{\varepsilon}}}_{2k}^{(n+1)}-
{\boldsymbol{\widehat{\varepsilon}}}_{2k}^{(n)}}{ <{\boldsymbol{\widehat{\varepsilon}}}^{(n+1)}_{2k+1}  -
{\boldsymbol{\widehat{\varepsilon}}}^{(n)}_{2k+1}, {\boldsymbol{\widehat{\varepsilon}}}^{(n+1)}_{2k}  -
{\boldsymbol{\widehat{\varepsilon}}}^{(n)}_{2k}>} \in E, \; \; k,n=0,1,\ldots
\end{array}
\right.
\end{equation}

Let us remind that the idea which led to the discovery of the two topological $\varepsilon$--algorithms for implementing the two versions of the topological Shanks sequence transformation was based on the definition of the inverse of a couple $(\mathbf u,\mathbf y)
\in E \times E^*$ defined as $\mathbf u^{-1}=\mathbf y/<\mathbf y,\mathbf u> \in E^*$ and $\mathbf y^{-1}=\mathbf u/<\mathbf y,\mathbf u> \in E$ \cite{etopo}. Keeping this definition in mind, we see that the inverses used for the
odd and even elements are, in fact, the inverses of the couple $({\boldsymbol{\widehat{\varepsilon}}}_{2k}^{(n+1)}-
{\boldsymbol{\widehat{\varepsilon}}}_{2k}^{(n)}, \mathbf y) \in E \times E^*$.

\vskip 2mm

The elements $\mathbf{\widetilde e_k}(\mathbf S_n) \in E$ of the second topological  Shanks transformation can be recursively computed by the
{\it second topological} $\varepsilon$--{\it algorithm} ({\sc tea{\small 2}}) \cite{etopo} whose rules are
(see \cite{cbmrz} for a {\sc fortran} subroutine)
\begin{equation}
\label{etop2}
\left\{
\begin{array}{lcl@{}}
\boldsymbol{\widetilde \varepsilon}_{-1}^{(n)} \!\!&\!=\!&\!\! \mathbf 0 \in E^*, \qquad n=0,1,\ldots,\\
\boldsymbol{\widetilde \varepsilon}_{0}^{(n)} \!\!&\!=\!&\!\! \mathbf S_n \in E, \qquad n=0,1,\ldots,\\
\boldsymbol{\widetilde \varepsilon}_{2k+1}^{(n)} \!\!&\!=\!&\!\!  \boldsymbol{\widetilde \varepsilon}_{2k-1}^{(n+1)}+
\displaystyle \frac{
{\mathbf{y}}}{<{\mathbf{y}},{\boldsymbol{\widetilde \varepsilon}}^{(n+1)}_{2k}
- {\boldsymbol{\widetilde \varepsilon}}^{(n)}_{2k}>}
\in E^*, \; \; k,n=0,1,\ldots,\\
\boldsymbol{\widetilde \varepsilon}_{2k+2}^{(n)} \!\!&\!=\!&\!\!  \boldsymbol{\widetilde \varepsilon}_{2k}^{(n+1)}\!+\!
\displaystyle \frac{
{\boldsymbol{\widetilde \varepsilon}}_{2k}^{(n+2)}-
{\boldsymbol{\widetilde \varepsilon}}_{2k}^{(n+1)}}{ <{\boldsymbol{\widetilde \varepsilon}}^{(n+1)}_{2k+1}  -
{\boldsymbol{\widetilde \varepsilon}}^{(n)}_{2k+1}, {\boldsymbol{\widetilde \varepsilon}}^{(n+2)}_{2k}  -
{\boldsymbol{\widetilde \varepsilon}}^{(n+1)}_{2k}>} \in E, \;  k,n=0,1,\ldots
\end{array}
\right.
\end{equation}

The following property holds
\begin{prop} {\em \cite{etopo}} \label{pro}
$$
\begin{array}{lll}
\boldsymbol{\widehat{\varepsilon}}_{2k}^{(n)}=\mathbf {\widehat{e}}_k(\mathbf S_n),  \qquad &
<\mathbf y, {\boldsymbol{\widehat{\varepsilon}}}^{(n)}_{2k}>=e_k(<\mathbf y, \mathbf S_n>),
 & \\
\boldsymbol{\widehat{\varepsilon}}_{2k+1}^{(n)}=
\mathbf y/<\mathbf y, \mathbf {\widehat{e}}_k(\Delta \mathbf S_n)>, \qquad &
{\boldsymbol{\widehat{\varepsilon}}}^{(n)}_{2k+1}=\mathbf y/ e_k(<\mathbf y, \Delta \mathbf S_n>),
\quad & k,n=0,1,\ldots.
\end{array}
$$
These relations are also true for the $\boldsymbol{\widetilde \varepsilon}_{k}^{(n)}$'s and the
$\mathbf {\widetilde{e}}_k$.
\end{prop}

\begin{rem} Due to the first rule of the algorithms, we have
$\boldsymbol{\widehat{\varepsilon}}_{2k+1}^{(n)}=\mathbf y \sum_{i=0}^k 1/<\mathbf y,
\boldsymbol{\widehat{\varepsilon}}_{2k}^{(n+1)}-
\boldsymbol{\widehat{\varepsilon}}_{2k}^{(n)}>$,
and the same relation for $\boldsymbol{\widetilde \varepsilon}_{2k+1}^{(n)}$, which shows
that these elements of $E^*$ are a multiple of $\mathbf y$.
\end{rem}

\begin{table}
{\small
$$
\begin{array}{|c@{~}c@{~}c@{~}c@{~}c||c@{~}c@{~}c@{~}c@{~}c|c@{~}c@{~}c@{~}c@{~}c|}
\multicolumn{5}{@{}c@{}}{\mbox{\small \rm Odd rule ({\sc tea1,
tea2})}} & \multicolumn{5}{@{}c@{}}{\mbox{\small \rm Even rule
({\sc tea1})}} &
\multicolumn{5}{@{}c}{\mbox{\small \rm Even rule  ({\sc tea2})}}\\
\hline &&&&&
&&&&&&&&&\\
&& \boldsymbol{{\varepsilon}}_{2k}^{(n)}&&&
\boldsymbol{\widehat{\varepsilon}}_{2k}^{(n)}&&&& &&&&&
\\
&\nearrow&&\searrow&& &\searrow&&& &&&&&
\\
\boldsymbol{{\varepsilon}}_{2k-1}^{(n+1)}&&&&
\fbox{$\boldsymbol{{\varepsilon}}_{2k+1}^{(n)}$}& &&
\boldsymbol{\widehat{\varepsilon}}_{2k+1}^{(n)}&&& &&
\boldsymbol{\widetilde \varepsilon}_{2k+1}^{(n)}&&
\\
&\searrow&&\nearrow&& &\nearrow&&\searrow&& &\nearrow&&\searrow&
\\
&& \boldsymbol{{\varepsilon}}_{2k}^{(n+1)}&&&
\boldsymbol{\widehat{\varepsilon}}_{2k}^{(n+1)}&&&&
\fbox{$\boldsymbol{\widehat{\varepsilon}}_{2k+2}^{(n)}$}&
\boldsymbol{\widetilde \varepsilon}_{2k}^{(n+1)}&&&&
\fbox{$\boldsymbol{\widetilde \varepsilon}_{2k+2}^{(n)}$}
\\
&&&&& &\searrow&&\nearrow&& &\searrow&&\nearrow&
\\
&&&&& && \boldsymbol{\widehat{\varepsilon}}_{2k+1}^{(n+1)}&&& &&
\boldsymbol{\widetilde \varepsilon}_{2k+1}^{(n+1)}&&
\\
&&&&& &&&&& &\nearrow&&&
\\
&&&&& &&&&& \boldsymbol{\widetilde \varepsilon}_{2k}^{(n+2)}&&&&
\\ &&&&&
&&&&&&&&&\\
\hline
\end{array}
$$
}
\caption{The relations for the first ({\sc tea1}) and the second
({\sc tea2}) topological $\varepsilon$--algorithms (in the odd
rule $\boldsymbol{{\varepsilon}}$ is
$\boldsymbol{\widehat{\varepsilon}}$ for {\sc tea1} and
$\boldsymbol{\widetilde{\varepsilon}}$ for {\sc tea2}).
\label{TEA12}}
\end{table}

If $E = \mathbb{R}$ or $\mathbb C$, the rules (\ref{etop}) and (\ref{etop2}) of the first
and the second topological $\varepsilon$--algorithms reduce to those of the scalar $\varepsilon$--algorithm.

%%%%% MRZ ADDED
The relationships between the elements of $E$ and $E^*$ in the
topological $\varepsilon$--array involved in these two topological
algorithms are showed in Table \ref{TEA12}. The odd rule has the
same structure as in the scalar $\varepsilon$--algorithm, but the
even ones need an additional term of the topological
$\varepsilon$--array.

 In the first and the second topological
$\varepsilon$--algorithm, the difficulty is to compute the duality
products appearing in their denominators since they involve linear
functionals which cannot be handled on a computer. This problem
will be solved by the new algorithms given in the next Section.

\section{The simplified topological $\boldsymbol\varepsilon$--algorithms}
\label{tsea}

When $E$ is an arbitrary vector space, the problem with the rules of the topological $\varepsilon$--algorithms,  is to be able to compute, for all $k$ and $n$, the duality product with the odd lower index terms $\boldsymbol{\widehat{\varepsilon}}_{2k+1}^{(n)}$'s and the $\boldsymbol{\widetilde\varepsilon}_{2k+1}^{(n)}$'s
on a computer. We will now derive new algorithms for implementing the two topological Shanks sequence transformations of Section \ref{tsst} which avoid the manipulation of elements of $E^*$, and the use of the duality product with $\mathbf y$ into the rules of the algorithms. In these new algorithms, the linear functional $\mathbf y$ will only have to be applied to the terms of the initial sequence $(\mathbf S_n)$.
Moreover, they require the storage of a fewer number of elements of $E$ and no element of $E^*$, since they only connect terms with an even lower index in the topological $\varepsilon$--array.
Their implementation will be described in Section \ref{impl}.

\subsection{The first simplified topological $\boldsymbol\varepsilon$--algorithm}

Applying Wynn's scalar $\varepsilon$--algorithm (\ref{esca}) to the sequence $(S_n=<\mathbf y, \mathbf S_n>)$,  we have, from Property \ref{pro}, $\varepsilon_{2k}^{(n)}= e_k(<\mathbf y, \mathbf S_n>)$ and, we obtain, by  Property \ref{esc},
\begin{eqnarray*}
{\boldsymbol{\widehat{\varepsilon}}}^{(n+1)}_{2k+1}-{\boldsymbol{\widehat{\varepsilon}}}^{(n)}_{2k+1} &=&
\frac{\mathbf y}{<\mathbf y, \mathbf {\widehat{e}}_k(\Delta \mathbf S_{n+1})>}-\frac{\mathbf y}{<\mathbf y, \mathbf {\widehat{e}}_k(\Delta \mathbf S_{n})>}\\
&=& \mathbf y \left( \frac{1}{ e_k(<\mathbf y,\Delta \mathbf S_{n+1}>)}- \frac{1}{ e_k(<\mathbf y,\Delta \mathbf S_{n}>)}\right)\\
&=& \mathbf y ( \varepsilon_{2k+1}^{(n+1)}- \varepsilon_{2k+1}^{(n)}).
\end{eqnarray*}

Thus, the last relation in (\ref{etop}) becomes

$$\boldsymbol{\widehat{\varepsilon}}_{2k+2}^{(n)} =  \boldsymbol{\widehat{\varepsilon}}_{2k}^{(n+1)}+
\frac{{\boldsymbol{\widehat{\varepsilon}}}_{2k}^{(n+1)}-{\boldsymbol{\widehat{\varepsilon}}}_{2k}^{(n)}}
{<\mathbf y, {\boldsymbol{\widehat{\varepsilon}}}^{(n+1)}_{2k}  -
{\boldsymbol{\widehat{\varepsilon}}}^{(n)}_{2k}> ( \varepsilon_{2k+1}^{(n+1)}- \varepsilon_{2k+1}^{(n)})}, \qquad k,n=0,1,\ldots .$$

Moreover, due to Property \ref{pro}, the quantities
$<\mathbf y, {\boldsymbol{\widehat{\varepsilon}}}^{(n+1)}_{2k}  -{\boldsymbol{\widehat{\varepsilon}}}^{(n)}_{2k}>$
can be computed by the scalar $\varepsilon$--algorithm, and we finally obtain the rule

\begin{equation}
\label{final}
\boldsymbol{\widehat{\varepsilon}}_{2k+2}^{(n)} =  \boldsymbol{\widehat{\varepsilon}}_{2k}^{(n+1)}+
\frac{1}
{( \varepsilon^{(n+1)}_{2k}  - \varepsilon^{(n)}_{2k})
( \varepsilon_{2k+1}^{(n+1)}- \varepsilon_{2k+1}^{(n)})}
({\boldsymbol{\widehat{\varepsilon}}}_{2k}^{(n+1)}-{\boldsymbol{\widehat{\varepsilon}}}_{2k}^{(n)}),
\; k,n=0,1,\ldots,
\end{equation}
with $\boldsymbol{\widehat{\varepsilon}}_{0}^{(n)}= \mathbf S_n \in E$, $n=0,1,\ldots$.

\vskip 2mm

This relation was already given in \cite{etopo} as Property 10 but not into an algorithmic form.
It shows that the first topological Shanks transformation can now be implemented by a triangular
scheme involving the scalar $\varepsilon$--algorithm and elements of $E$ exclusively, instead of
extended rhombus rules needing the duality product with elements of $E^*$. Moreover, we now have
only one rule instead of two.

Thanks to the recursive rule of the scalar $\varepsilon$--algorithm, the algorithm (\ref{final})
can also be written under one of the following equivalent forms

\begin{eqnarray}
\boldsymbol{\widehat{\varepsilon}}_{2k+2}^{(n)} &=&  \boldsymbol{\widehat{\varepsilon}}_{2k}^{(n+1)}+
\frac{ \varepsilon_{2k+1}^{(n)}- \varepsilon_{2k-1}^{(n+1)}}
{ \varepsilon_{2k+1}^{(n+1)}- \varepsilon_{2k+1}^{(n)}}
({\boldsymbol{\widehat{\varepsilon}}}_{2k}^{(n+1)}-{\boldsymbol{\widehat{\varepsilon}}}_{2k}^{(n)}),
\label{final1}\\
\boldsymbol{\widehat{\varepsilon}}_{2k+2}^{(n)} &=&  \boldsymbol{\widehat{\varepsilon}}_{2k}^{(n+1)}+
\frac{ \varepsilon^{(n)}_{2k+2}  - \varepsilon^{(n+1)}_{2k}}
{ \varepsilon^{(n+1)}_{2k}  - \varepsilon^{(n)}_{2k}}
({\boldsymbol{\widehat{\varepsilon}}}_{2k}^{(n+1)}-{\boldsymbol{\widehat{\varepsilon}}}_{2k}^{(n)}),
\label{final2}\\
\boldsymbol{\widehat{\varepsilon}}_{2k+2}^{(n)} &=&  \boldsymbol{\widehat{\varepsilon}}_{2k}^{(n+1)}+
( \varepsilon_{2k+1}^{(n)}- \varepsilon_{2k-1}^{(n+1)})
( \varepsilon^{(n)}_{2k+2}  - \varepsilon^{(n+1)}_{2k})
({\boldsymbol{\widehat{\varepsilon}}}_{2k}^{(n+1)}-{\boldsymbol{\widehat{\varepsilon}}}_{2k}^{(n)}).
\label{final3}
\end{eqnarray}

Notice that (\ref{final2}) can also be written
\begin{eqnarray}
{\boldsymbol{\widehat{\varepsilon}}}_{2k+2}^{(n)}&=&\frac{\varepsilon^{(n)}_{2k+2}  - \varepsilon^{(n)}_{2k}}{\varepsilon^{(n+1)}_{2k}  - \varepsilon^{(n)}_{2k}}{\boldsymbol{\widehat{\varepsilon}}}_{2k}^{(n+1)}-\frac{\varepsilon^{(n)}_{2k+2}  - \varepsilon^{(n+1)}_{2k}}{\varepsilon^{(n+1)}_{2k}  - \varepsilon^{(n)}_{2k}}{\boldsymbol{\widehat{\varepsilon}}}_{2k}^{(n)} \label{final21}\\
&=&{\boldsymbol{\widehat{\varepsilon}}}_{2k}^{(n)}+\frac{\varepsilon^{(n)}_{2k+2}  - \varepsilon^{(n)}_{2k}}{\varepsilon^{(n+1)}_{2k}  - \varepsilon^{(n)}_{2k}}({\boldsymbol{\widehat{\varepsilon}}}_{2k}^{(n+1)}-
{\boldsymbol{\widehat{\varepsilon}}}_{2k}^{(n)}).\label{final22}
\end{eqnarray}

This new algorithm, in any of the preceding forms, is called the
{\it first simplified topological} $\varepsilon$--{\it algorithm}, and it is denoted by {\sc stea{\small 1}}.

Taking the duality product of $\mathbf y$ with any of the preceding relations (\ref{final}--\ref{final22}) allows to recover the rule of the scalar $\varepsilon$--algorithm or leads to an identity. The same is true when $E$ is $\mathbb R$ or
$\mathbb C$.

 From (\ref{final21}), we have the

\begin{prop}
The computation of ${\boldsymbol{\widehat{\varepsilon}}}_{2k+2}^{(n)}$ is stable for all $n$ if $\exists M_k$, independent of $n$,
such that
$$\left|\frac{\varepsilon^{(n)}_{2k+2}  - \varepsilon^{(n)}_{2k}}{\varepsilon^{(n+1)}_{2k}  - \varepsilon^{(n)}_{2k}}\right|+
\left|\frac{\varepsilon^{(n)}_{2k+2}  - \varepsilon^{(n+1)}_{2k}}{\varepsilon^{(n+1)}_{2k}  - \varepsilon^{(n)}_{2k}}\right|
\leq M_k.$$
\end{prop}

\begin{rem}
Although the $\boldsymbol{\widehat{\varepsilon}}_{2k+1}^{(n)}$'s are no longer needed,  it holds from Property \ref{pro} that
$\boldsymbol{\widehat{\varepsilon}}_{2k+1}^{(n)} =  \boldsymbol{\widehat{\varepsilon}}_{2k-1}^{(n+1)}+
{\mathbf{y}}/( \varepsilon^{(n+1)}_{2k}- \varepsilon^{(n)}_{2k})$, and thus, similarly to Property \ref{prop1}, we obtain
$\boldsymbol{\widehat{\varepsilon}}_{2k+1}^{(n)}= {\bf y} \sum_{i=0}^k 1/\Delta \varepsilon_{2i}^{(n+k-i)}$.
Since $\Delta S_n=<\mathbf y, \Delta \mathbf S_n>$, we also have, from what precedes and Property \ref{pro},
${\boldsymbol{\widehat{\varepsilon}}}^{(n)}_{2k+1}=\mathbf y   \varepsilon_{2k+1}^{(n)}$ and $\varepsilon_{2k+1}^{(n)}=1/ e_k(<\mathbf y, \Delta \mathbf S_n>)$.
These relations also hold for the $\boldsymbol{\widetilde\varepsilon}_{2k+2}^{(n)}$'s of the second simplified topological $\varepsilon$--algorithm (see below).
\end{rem}

\subsection{The second simplified topological $\boldsymbol\varepsilon$--algorithm}
\label{STEAsection}
Since similar algebraic properties hold for the second topological Shanks transformation
and the second topological $\varepsilon$--algorithm, we can similarly derive a
{\it second simplified topological $\varepsilon$--algorithm}, denoted by {\sc stea{\small 2}}.
Its rules can be written in any of the four following equivalent forms

\begin{eqnarray*}
\boldsymbol{\widetilde\varepsilon}_{2k+2}^{(n)} &=&  \boldsymbol{\widetilde\varepsilon}_{2k}^{(n+1)}+
\frac{1}
{( \varepsilon^{(n+2)}_{2k}  - \varepsilon^{(n+1)}_{2k})
( \varepsilon_{2k+1}^{(n+1)}- \varepsilon_{2k+1}^{(n)})}
({\boldsymbol{\widetilde\varepsilon}}_{2k}^{(n+2)}-{\boldsymbol{\widetilde\varepsilon}}_{2k}^{(n+1)}), \\
\boldsymbol{\widetilde\varepsilon}_{2k+2}^{(n)} &=&  \boldsymbol{\widetilde\varepsilon}_{2k}^{(n+1)}+
\frac{ \varepsilon_{2k+1}^{(n+1)}- \varepsilon_{2k-1}^{(n+2)}}
{ \varepsilon_{2k+1}^{(n+1)}- \varepsilon_{2k+1}^{(n)}}
({\boldsymbol{\widetilde\varepsilon}}_{2k}^{(n+2)}-{\boldsymbol{\widetilde\varepsilon}}_{2k}^{(n+1)}),\\
\boldsymbol{\widetilde\varepsilon}_{2k+2}^{(n)} &=&  \boldsymbol{\widetilde\varepsilon}_{2k}^{(n+1)}+
\frac{ \varepsilon^{(n)}_{2k+2}  - \varepsilon^{(n+1)}_{2k}}
{ \varepsilon^{(n+2)}_{2k}  - \varepsilon^{(n+1)}_{2k}}
({\boldsymbol{\widetilde\varepsilon}}_{2k}^{(n+2)}-{\boldsymbol{\widetilde\varepsilon}}_{2k}^{(n+1)}),\\
{\boldsymbol{\widetilde\varepsilon}_{2k+2}}^{(n)} &=&  \boldsymbol{\widetilde\varepsilon}_{2k}^{(n+1)}+
( \varepsilon_{2k+1}^{(n+1)}- \varepsilon_{2k-1}^{(n+2)})
( \varepsilon^{(n)}_{2k+2}  - \varepsilon^{(n+1)}_{2k})
({\boldsymbol{\widetilde\varepsilon}}_{2k}^{(n+2)}-{\boldsymbol{\widetilde\varepsilon}}_{2k}^{(n+1)}),
\end{eqnarray*}
with $\boldsymbol{\widetilde{\varepsilon}}_{0}^{(n)}= \mathbf S_n \in E$, $n=0,1,\ldots$.

In these formulae, the scalar quantities are the same as those used in the first simplified topological $\varepsilon$--algorithm, that is they are obtained by applying the scalar $\varepsilon$--algorithm
to the sequence $(<\mathbf y, \mathbf S_n>)$.

\begin{table}[htb]
\[
\begin{array}{|ccc|ccc|}
\multicolumn{3}{@{}c@{}}{\mbox{\rm {\small \sc stea1}}} &
\multicolumn{3}{@{}c@{}}{\mbox{\rm  {\small \sc stea2}}} \\
\hline &&&&&\\
\boldsymbol{\widehat{\varepsilon}}_{2k}^{(n)}&& &&&\\
&\searrow&&&&\\
\boldsymbol{\widehat{\varepsilon}}_{2k}^{(n+1)}&\longrightarrow&
\fbox{$\boldsymbol{\widehat{\varepsilon}}_{2k+2}^{(n)}$} &
\boldsymbol{\widetilde\varepsilon}_{2k}^{(n+1)}&\longrightarrow&
\fbox{$\boldsymbol{\widetilde\varepsilon}_{2k+2}^{(n)}$}\\
&&&&\nearrow&\\
&&&\boldsymbol{\widetilde\varepsilon}_{2k}^{(n+2)}&&\\
 &&&&&\\
 \hline
\end{array}
\]
\caption{The relations for the first ({\small \sc stea1}) and the
second ({\small \sc stea2}) simplified topological
$\varepsilon$--algorithms. \label{STEA12} }
\end{table}

%%%%% MRZ ADDED
In Table \ref{STEA12}, we show the very simple triangular
relationships between the elements of $E$ in the topological
$\varepsilon$--array for the two simplified algorithms. Remark
that only the even terms are needed and computed.

\begin{rem}
Since the $\boldsymbol{{\widehat{\varepsilon}}}_{2k}^{(n)}$ and the $\boldsymbol{\widetilde\varepsilon}_{2k}^{(n)}$ are related by a triangular recursive scheme, they satisfy the theory of {\it reference functionals} developed in \cite{cbgw}.
\end{rem}

\section{Convergence and acceleration}
\label{conacc}

Several convergence and acceleration results for the topological $\varepsilon$--algorithm were already given in \cite{etopo}. However, the conditions of these theorems are difficult to check since they involved elements of $E^*$ which have to be recursively computed.
The fact that, now, the simplified algorithms no longer involved these linear functionals will allow us to give new results, easier to verify in practice. Thus, the simplified $\varepsilon$--algorithms not only play a major role in the implementation of the topological Shanks transformations, but they also have a primary impact on the theoretical results that can be proved.

\vskip 2mm

Property \ref{prp} means that the numbers $e_k(S_n)$ obtained by applying the scalar Shanks transformation to the sequence of numbers $(S_n=<\mathbf y, \mathbf S_n>)$ and the elements $\mathbf {\widehat{e}}_k(\mathbf S_n)$
and $\mathbf {\widetilde{e}}_k(\mathbf S_n)$ of $E$, obtained by the topological Shanks transformations applied to the sequence $(\mathbf S_n)$ of elements of $E$, are computed by a linear combination with the same coefficients. This fact implies that, under some circumstances, the sequences $(e_k(S_n))$, $(\mathbf {\widehat{e}}_k(\mathbf S_n))$, and $(\mathbf {\widetilde{e}}_k(\mathbf S_n))$ can share similar convergence and acceleration behaviors, and we will now illustrate this fact.

\vskip 2mm

Let us begin by two results about important classes of sequences which, despite of their simplicity, remained unnoticed.

\begin{dfn}
A sequence of vectors in $\mathbb R^m$ or matrices in $\mathbb R^{m \times s}$ is {\em totally monotonic}, and we write $({\mathbf S_n}) \in \mbox{TM}$, if, for all $k$ and $n$, $(-1)^k \Delta^k {\mathbf S_n} \geq \mathbf 0$, where the inequality has to be understood for each component of the vectors or each element of the matrices. A sequence of vectors or matrices is {\em totally oscillating}, and we write $({\mathbf S_n}) \in \mbox{TO}$, if $((-1)^n {\mathbf S_n}) \in \mbox{TM}$.
\end{dfn}

In the scalar case, these sequences were first studied by Wynn \cite{pwconv} and his results were complemented by Brezinski \cite{etude,cbcr}. Their construction and their properties were studied in \cite{gene,pwcras}.

\vskip 2mm

Since the coefficients $a_i^{(n,k)}$ in (\ref{scTEA}) and (\ref{sTEA}) are the same, then, if $({\mathbf S_n}) \in \mbox{TM or TO}$ and if $(S_n=<\mathbf y,{\mathbf S_n}>) \in \mbox{TM or TO}$, then the scalars $\varepsilon_{2k}^{(n)}$, the components of the vectors or the elements of the matrices ${\boldsymbol{\widehat{\varepsilon}}}_{2k}^{(n)}$, and those of ${\boldsymbol{\widetilde\varepsilon}}_{2k}^{(n)}$ satisfy the same inequalities. Thus we have the following new results

\begin{theorem}
If $(\mathbf S_n)$ converges to $\mathbf S$, if $(S_n=<\mathbf y,{\mathbf S_n}>)
\in \mbox{TM}$, and if $\exists a \neq 0$ and ${\mathbf b} \in \mathbb R^m$
(or $\mathbb R^{m \times s}$ in the matrix case) such that
$(a\mathbf S_n+{\mathbf b}) \in \mbox{TM}$, then
$$
\begin{array}{ll}
\mathbf 0 \leq a
{\boldsymbol{\widehat{\varepsilon}}}_{2k+2}^{(n)}+{\mathbf b} \leq
a {\boldsymbol{\widehat{\varepsilon}}}_{2k}^{(n)}+{\mathbf b}, &
\mathbf 0 \leq a {\boldsymbol{\widehat{\varepsilon}}}_{2k}^{(n+1)}+
{\mathbf b} \leq a {\boldsymbol{\widehat{\varepsilon}}}_{2k}^{(n)}+{\mathbf b},\\
\mathbf 0 \leq a {\boldsymbol{\widehat{\varepsilon}}}_{2k+2}^{(n)}+{\mathbf b}
\leq a {\boldsymbol{\widehat{\varepsilon}}}_{2k}^{(n+1)}+{\mathbf b}, &
\mathbf 0 \leq a {\boldsymbol{\widehat{\varepsilon}}}_{2k+2}^{(n)}+
{\mathbf b} \leq a {\boldsymbol{\widehat{\varepsilon}}}_{2k}^{(n+2)}+{\mathbf b},\\
\displaystyle \forall k, \mbox{~fixed}, \quad \lim_{n \to \infty}
{\boldsymbol{\widehat{\varepsilon}}}_{2k}^{(n)}=\mathbf S, &
\displaystyle
\forall n, \mbox{~fixed}, \quad \lim_{k \to \infty}
{\boldsymbol{\widehat{\varepsilon}}}_{2k}^{(n)}=\mathbf S.
\end{array}
$$
 Moreover, if  $\lim_{n \to \infty} <\mathbf y, \mathbf
S_{n+1}-\mathbf S>/<\mathbf y, \mathbf S_{n}-\mathbf S> \neq 1$,
then
\begin{eqnarray*}
&&\forall k, ~ \mbox{fixed}, \quad \lim_{n \to \infty} \|{\boldsymbol{\widehat{\varepsilon}}}_{2k}^{(n)}-\mathbf S\|/\|\mathbf S_{n+2k}-\mathbf S\|=0\\
&&\forall n, ~ \mbox{fixed}, \quad \lim_{k \to \infty} \|{\boldsymbol{\widehat{\varepsilon}}}_{2k}^{(n)}-\mathbf S\|/\|\mathbf S_{n+2k}-\mathbf S\|=0.
\end{eqnarray*}
The same results hold for the ${\boldsymbol{\widetilde\varepsilon}}_{2k}^{(n)}$.
\end{theorem}

\begin{theorem}
If $(\mathbf S_n)$ converges to $\mathbf S$, if $(S_n=<\mathbf y,{\mathbf S_n}>) \in \mbox{TO}$, and if $\exists a \neq 0$ and ${\mathbf b} \in \mathbb R^m$ (or $\mathbb R^{m \times s}$ in the matrix case) such that
$(a\mathbf S_n+{\mathbf b}) \in \mbox{TO}$, then
$$
\begin{array}{ll}
\mathbf 0 \leq a {\boldsymbol{\widehat{\varepsilon}}}_{2k+2}^{(2n)}+
{\mathbf b} \leq a {\boldsymbol{\widehat{\varepsilon}}}_{2k}^{(2n)}+{\mathbf b}, &
a({\boldsymbol{\widehat{\varepsilon}}}_{2k}^{(2n+1)}-
{\boldsymbol{\widehat{\varepsilon}}}_{2k}^{(2n)}) \leq a
({\boldsymbol{\widehat{\varepsilon}}}_{2k+2}^{(2n+1)}-
{\boldsymbol{\widehat{\varepsilon}}}_{2k+2}^{(2n)}) \leq \mathbf 0,\\
a {\boldsymbol{\widehat{\varepsilon}}}_{2k}^{(2n+1)}+{\mathbf b}
\leq a {\boldsymbol{\widehat{\varepsilon}}}_{2k+2}^{(2n+1)}+{\mathbf b} \leq \mathbf 0, &
\mathbf 0 \leq a({\boldsymbol{\widehat{\varepsilon}}}_{2k+2}^{(2n+2)}-
{\boldsymbol{\widehat{\varepsilon}}}_{2k+2}^{(2n+1)}) \leq a
({\boldsymbol{\widehat{\varepsilon}}}_{2k}^{(2n+2)}-
{\boldsymbol{\widehat{\varepsilon}}}_{2k}^{(2n+1)}),
\\
 \mathbf 0 \leq a {\boldsymbol{\widehat{\varepsilon}}}_{2k+2}^{(2n)}+{\mathbf b} \leq a
 {\boldsymbol{\widehat{\varepsilon}}}_{2k}^{(2n+2)}+{\mathbf b}, &
 \displaystyle \forall k,  \mbox{~fixed}, \quad \lim_{n \to \infty}
{\boldsymbol{\widehat{\varepsilon}}}_{2k}^{(n)}=\mathbf S,
 \\
a {\boldsymbol{\widehat{\varepsilon}}}_{2k}^{(2n+3)}+{\mathbf b} \leq a
{\boldsymbol{\widehat{\varepsilon}}}_{2k+2}^{(2n+1)}+{\mathbf b} \leq \mathbf 0, &
\displaystyle \forall n,  \mbox{~fixed}, \quad \lim_{k \to \infty}
{\boldsymbol{\widehat{\varepsilon}}}_{2k}^{(n)}=\mathbf S.
\end{array}
$$
Moreover
\begin{eqnarray*}
&&\forall k, ~ \mbox{fixed}, \quad \lim_{n \to \infty} \|{\boldsymbol{\widehat{\varepsilon}}}_{2k}^{(n)}-\mathbf S\|/\|\mathbf S_{n+2k}-\mathbf S\|=0\\
&&\forall n, ~ \mbox{fixed}, \quad \lim_{k \to \infty} \|{\boldsymbol{\widehat{\varepsilon}}}_{2k}^{(n)}-\mathbf S\|/\|\mathbf S_{n+2k}-\mathbf S\|=0.
\end{eqnarray*}
The same results hold for the ${\boldsymbol{\widetilde\varepsilon}}_{2k}^{(n)}$.
\end{theorem}

\begin{rem}
If the vectors ${\mathbf S_n}$ belong to TM or TO and if $\mathbf y \geq \mathbf 0$, then $(S_n=<\mathbf y,{\mathbf S_n}>) \in \mbox{TM or TO}$. If the matrices ${\mathbf S_n}$ belong to TM and if $\mathbf y$ is the linear form such that, for any matrix $ \mathbf M$, $<\mathbf y, \mathbf M>=$trace$(\mathbf M)$, then $(S_n=<\mathbf y,{\mathbf S_n}>) \in \mbox{TM or TO}$.

As noticed in \cite{midy}, if, for all $n$, $\mathbf S_n \in \mathbb R^m$ is replaced by $ \mathbf D \mathbf S_n$, where $ \mathbf D$ is a diagonal regular matrix of dimension $m$, and if $\mathbf y$ is replaced by
$\lambda  \mathbf D^{-1} \mathbf y$, where $\lambda$ is a nonzero scalar, then, for all $k$ and $n$,
$\boldsymbol{\widehat{\varepsilon}}_{2k}^{(n)}$ becomes $ \mathbf D\boldsymbol{\widehat{\varepsilon}}_{2k}^{(n)}$ and
$\boldsymbol{\widehat{\varepsilon}}_{2k+1}^{(n)}$ becomes $ \mathbf D^{-1}\boldsymbol{\widehat{\varepsilon}}_{2k+1}^{(n)}$.
The same property holds for the $\boldsymbol{\widetilde{\varepsilon}}_{k}^{(n)}$'s.This remark allows to extend the two previous theorems.
\end{rem}

Let us continue by some other convergence and acceleration results that would not have been easily obtained directly from the determinantal formulae of the topological Shanks transformations or from the rules of the topological $\varepsilon$--algorithms.

\vskip 2mm

Assume now that $E$ is a normed vector space and that ${\mathbf y} \in E'\subseteq E^*$, where
 $E'$ is the space of all continuous linear functionals on $E$. We will only consider the case of the first transformation and the first algorithm, the second ones could be treated similarly and led to equivalent results.

\vskip 2mm

Let us set
\begin{equation}
\label{eq11}
r_k^{(n)}=(\varepsilon^{(n)}_{2k+2}  - \varepsilon^{(n+1)}_{2k})/(\varepsilon^{(n+1)}_{2k}  - \varepsilon^{(n)}_{2k}).
\end{equation}
The rule (\ref{final2}) of the first simplified topological $\varepsilon$--algorithm gives
$$\left\|\boldsymbol{\widehat{\varepsilon}}_{2k+2}^{(n)} -  \boldsymbol{\widehat{\varepsilon}}_{2k}^{(n+1)}\right\|/
\left\|{\boldsymbol{\widehat{\varepsilon}}}_{2k}^{(n+1)}-{\boldsymbol{\widehat{\varepsilon}}}_{2k}^{(n)}\right\|=|r_k^{(n)}|,$$
which shows that both ratios have the same form and behave similarly. We have the following convergence result which, in fact, is nothing else than Toeplitz theorem for a summation process

\begin{theorem}
\label{thm1}
If $\lim_{n \to \infty}{\boldsymbol{\widehat{\varepsilon}}}_{2k}^{(n)}=\bf S$, and if $\exists M$ such that,
$\forall n \geq N$, $|r_k^{(n)}| \leq M$, then
$\lim_{n \to \infty}{\boldsymbol{\widehat{\varepsilon}}}_{2k+2}^{(n)}=\bf S$.
\end{theorem}

Thus, the convergence of the first simplified topological $\varepsilon$--algorithm depends on the behavior of the scalar one,
and the choice of $\mathbf y$ intervenes in the conditions to be satisfied. The other forms of the first simplified $\varepsilon$--algorithm lead to different ratios linking its behavior with that of the scalar $\varepsilon$--algorithm.

\vskip 2mm

Let us now give an acceleration result in the case where $E=\mathbb R^m$, $m>1$,

\begin{theorem}
\label{thm0}
Assume that $E=\mathbb R^m$, that $r_k^{(n)}=r_k+o(1)$ with $r_k \neq -1$, and that $\forall i,  ({\boldsymbol{\widehat{\varepsilon}}}_{2k}^{(n+1)}-\mathbf S)_i/({\boldsymbol{\widehat{\varepsilon}}}_{2k}^{(n)}-\mathbf S)_i=r_k/(1+r_k)+o(1)$, then $\lim_{n \to \infty} \|{\boldsymbol{\widehat{\varepsilon}}}_{2k+2}^{(n)}-\mathbf S\|
/\|{\boldsymbol{\widehat{\varepsilon}}}_{2k}^{(n)}-\mathbf S\|=0$.
\end{theorem}
\begin{proof}
Writing  (\ref{final2}) as (which is the same as (\ref{final21}))
$${\boldsymbol{\widehat{\varepsilon}}}_{2k+2}^{(n)}-\mathbf S=(1+r_k^{(n)})({\boldsymbol{\widehat{\varepsilon}}}_{2k}^{(n+1)}-\mathbf S)-r_k^{(n)}
({\boldsymbol{\widehat{\varepsilon}}}_{2k}^{(n)}-\mathbf S),$$ we have
from the assumptions on $r_k^{(n)}$ and on the components of the vectors
$$({\boldsymbol{\widehat{\varepsilon}}}_{2k+2}^{(n)}-\mathbf S)_i=[(1+r_k+o(1))(r_k/(1+r_k)+o(1))-(r_k+o(1))]
({\boldsymbol{\widehat{\varepsilon}}}_{2k}^{(n)}-\mathbf S)_i=o(1)({\boldsymbol{\widehat{\varepsilon}}}_{2k}^{(n)}-\mathbf S)_i.$$
Thus the result since all norms are equivalent.
\end{proof}

\begin{rem}
By the assumptions on the components of the vectors, it holds
$\lim_{n \to \infty}(\varepsilon^{(n+1)}_{2k}-S)/(\varepsilon^{(n)}_{2k}-S)=r_k/(1+r_k).$ Since $r_k^{(n)}=r_k+o(1)$ and
$$r_k^{(n)}=
\left(\displaystyle
\frac{\varepsilon^{(n)}_{2k+2}-S}{\varepsilon^{(n)}_{2k}-S}
- \frac{\varepsilon^{(n+1)}_{2k}-S}{\varepsilon^{(n)}_{2k}-S}\right)\Big/
\left({ \displaystyle
\frac{\varepsilon^{(n+1)}_{2k}-S} {\varepsilon^{(n)}_{2k}-S}-1}\right),
$$
it follows that we also have $\lim_{n \to \infty}(\varepsilon^{(n)}_{2k+2}-S)/(\varepsilon^{(n)}_{2k}-S)=0$.
\end{rem}

\vskip 2mm

Let us now give convergence and acceleration results concerning four special types of sequences. These sequences were considered by Wynn in the scalar case and for the scalar $\varepsilon$--algorithm \cite{pwconv}. We will see that his results can be extended to similar sequences of elements of $E$. The first result is the following

\begin{theorem}
\label{thm2}
We consider sequences of the form
$$\mbox{a)~} \mathbf{S_n-S}  \sim \sum_{i=1}^\infty a_i\lambda_i^n \mathbf{u_i},\quad (n \to \infty)
\mbox{~~or~~} \mbox{b)~} \mathbf{S_n-S}  \sim (-1)^n \sum_{i=1}^\infty a_i\lambda_i^n \mathbf{u_i},
\quad (n \to \infty),
$$
where $a_i, \lambda_i \in \mathbb K$, $\mathbf{u_i} \in E$, and $1>\lambda_1>\lambda_2>\cdots>0$. Then, when $k$ is fixed and $n$ tends to infinity,
$$
\begin{array}{ll}
\displaystyle
{\boldsymbol{\widehat{\varepsilon}}}_{2k}^{(n)}-\mathbf S={\cal O}(\lambda_{k+1}^n), & \quad
\displaystyle \frac{\|{\boldsymbol{\widehat{\varepsilon}}}_{2k+2}^{(n)}-\mathbf S\|}
{\|{\boldsymbol{\widehat{\varepsilon}}}_{2k}^{(n)}-\mathbf S\|}={\cal O}((\lambda_{k+2}/\lambda_{k+1})^n).
\end{array}
$$
\end{theorem}
\begin{proof}
We consider sequences of the form \mbox{\it a)}. We have
$$<\mathbf y,\mathbf S_n>-<\mathbf y,\mathbf S> \sim \sum_{i=1}^\infty a_i\lambda_i^n <\mathbf y,\mathbf{u_i}>.$$
If the scalar $\varepsilon$--algorithm is applied to the sequence $(<\mathbf y,\mathbf S_n>)$,
Wynn \cite{pwconv} proved that
$$\varepsilon_{2k}^{(n)}-<\mathbf y,\mathbf S> = a_{k+1}\frac
{(\lambda_{k+1}-\lambda_{1})^2 \cdots (\lambda_{k+1}-\lambda_{k})^2}
{(1-\lambda_{1})^2 \cdots (1-\lambda_{k})^2}\lambda_{k+1}^n <\mathbf y,\mathbf{u_{k+1}}>+{\cal O}(\lambda_{k+2}^n).$$
After some algebraic manipulations using (\ref{eq11}), we get
\begin{eqnarray*}
r_k^{(n)} & \! \sim \! & \frac{\lambda_{k+1}}{1-\lambda_{k+1}}\!+\!\frac{a_{k+2} <\mathbf y,\mathbf u_{k+2}>(\lambda_{k+2}-\lambda_{1})^2
\cdots (\lambda_{k+2}-\lambda_{k+1})^2}
{a_{k+1} <\mathbf y,\mathbf u_{k+1}>(\lambda_{k+1}-\lambda_{1})^2
\cdots (\lambda_{k+1}-\lambda_{k})^2(1-\lambda_{k+1})^3}\!\!\left(\frac{\lambda_{k+2}}{\lambda_{k+1}}\right)^{\!\!n} \\
&& \mbox{~~~} +
o\left(\frac{\lambda_{k+2}}{\lambda_{k+1}}\right)^{\!\!n}\!\!.
\end{eqnarray*}
Assume that
$${\boldsymbol{\widehat{\varepsilon}}}_{2k}^{(n)}-\mathbf S \sim \sum_{i=k+1}^p \alpha_{k,i} \lambda_i^n \mathbf u_i,$$
where the coefficients $\alpha_{k,i}$ depend on $k$ and $i$ but not on $n$
(by assumption, this is true for $k=0$).

Plugging this relation into the rule (\ref{final2}) of the first simplified topological $\varepsilon$--algorithm, and using the preceding expression for $r_k^{(n)}$, we obtain (see \cite{epve}, for a similar proof) that
$${\boldsymbol{\widehat{\varepsilon}}}_{2k+2}^{(n)}-\mathbf S \sim \frac{1}{1-\lambda_{k+1}}\left(\sum_{i=k+1}^p \alpha_{k,i} \lambda_i^{n+1} \mathbf u_i-\lambda_{k+1}\sum_{i=k+1}^p \alpha_{k,i} \lambda_i^{n} \mathbf u_i\right),$$
which proves the first result by induction. The second result follows immediately.

\vskip 1mm

\noindent The proof for sequences of the form \mbox{\it b)} is similar
to the preceding case by replacing each $\lambda_i$ by $-\lambda_i$.
\end{proof}

An application of this result will be given in Section \ref{linit}.

\vskip 1mm

The next theorem, whose proof is obtained by expressing $r_k^{(n)}$  as above from the result obtained by Wynn in the scalar case \cite{pwconv}, is
\begin{theorem}
\label{thm4}
We consider sequences of the form
$$\mathbf{S_n-S}  \sim \sum_{i=1}^\infty a_i (n+b)^{-i}  \mathbf{u_i},\quad (n \to \infty),$$
where $a_i, b \in \mathbb K$, $\mathbf{u_i} \in E$. Then, when $k$ is fixed and $n$ tends to infinity,
$${\boldsymbol{\widehat{\varepsilon}}}_{2k}^{(n)}-\mathbf S \sim \frac{a_1 \mathbf{u_1}}{(k+1)(n+b)}.$$
\end{theorem}
Thus, when $n$ tends to infinity, the sequence $({\boldsymbol{\widehat{\varepsilon}}}_{2k+2}^{(n)})$ does not converge to
$\mathbf S$ faster than the sequence $({\boldsymbol{\widehat{\varepsilon}}}_{2k}^{(n)})$.
We finally have the last result whose proof uses (\ref{final22}) and the expression given by Wynn \cite{pwconv} in the scalar case
\begin{theorem}
\label{thm5}
We consider sequences of the form
$$\mathbf{S_n-S}  \sim (-1)^n \sum_{i=1}^\infty a_i  (n+b)^{-i} \mathbf{u_i},\quad (n \to \infty),$$
where $a_i, b \in \mathbb K$, $\mathbf{u_i} \in E$. Then, when $k$ is fixed and $n$ tends to infinity
$$
\begin{array}{ll}
\displaystyle
{\boldsymbol{\widehat{\varepsilon}}}_{2k}^{(n)}-\mathbf S \sim (-1)^n \frac{a_1
(k\:!)^2 \mathbf{u_1}}{4^k(n+b)^{2k+1}}, & \quad
\displaystyle \frac{\|{\boldsymbol{\widehat{\varepsilon}}}_{2k+2}^{(n)}-\mathbf S\|}
{\|{\boldsymbol{\widehat{\varepsilon}}}_{2k}^{(n)}-\mathbf S\|}={\cal O}(1/(n+b)).
\end{array}
$$
\end{theorem}

The results of \cite{gari} about sequences with more complicated asymptotic expansions could possibly be extended similarly to the first simplified topological $\varepsilon$--algorithm. Since the vector $E$--algorithm \cite{ealg} can be used for implementing the first topological Shanks transformation, the acceleration results proved for it should also be valid for the first simplified topological $\varepsilon$--algorithm \cite{matos}.

\section{Implementation and performances}
\label{impl}

For implementing all the $\varepsilon$--alg\-orith\-ms
and for constructing the $\varepsilon$--array,
the simplest procedure is to store all its elements, for all $k$ and $n$.
Starting from a given number of terms in the two initial columns (the second one for the simplified
algorithms), the other
columns are computed one by one, each of them having one term less than the preceding one (two terms less for the simplified
algorithms).
Thus, a triangular part of the $\varepsilon$--array  is obtained (only the even columns in the
simplified cases). However, such a procedure
requires to store all the elements of this triangular part, and it can be costly for sequences of  vectors
or matrices.

A better procedure, which avoids to store the whole triangular array,
is to add each new term of the original sequence one by one, and to proceed with the computation of the
$\varepsilon$--array by ascending diagonal. This technique was invented   %%%%%
by Wynn for the scalar $\varepsilon$--algorithm and other algorithms having a
similar type of rules \cite{asc1,asc2} (see also \cite[pp. 397ff.]{cbmrz}), including the vector $\varepsilon$--algorithm \cite{evec}. It can be summarized as follows:
after having computed a triangular part of the $\varepsilon$--array,
and stored only its last ascending diagonal (containing all the terms needed for computing
the next diagonal), a new element of the initial sequence is introduced
and the next ascending diagonal is computed, element by element.
This technique only requires the storage of one ascending diagonal and three temporary auxiliary elements.
Let us mention that, for computing the descending diagonal $\varepsilon_0^{(0)}, \varepsilon_1^{(0)},\varepsilon_2^{(0)},
\ldots$, the new ascending diagonal has to be computed as far as possible, while, for computing the descending column
$\varepsilon_{2k}^{(0)},\varepsilon_{2k}^{(1)},\varepsilon_{2k}^{(2)},\ldots$, for a fixed value of $k$, the new ascending diagonal has to be computed until the column $2k$ has been reached.

In the topological $\varepsilon$--algorithms, the odd rule has the same form as that of the scalar
$\varepsilon$--algorithm, but the even rule needs an extra element, as showed in
Table \ref{TEA12}. This fact is not a problem
for implementing {\sc tea{\small 2}} with exactly the same technique as Wynn's, since the extra term,
namely $\boldsymbol{\widetilde{\varepsilon}}_{2k}^{(n+2)}$, needed for computing
$\boldsymbol{\widetilde{\varepsilon}}_{2k+2}^{(n)}$,  is an already computed element of the
diagonal we are constructing.
Thus, it is necessary to store only the preceding diagonal consisting of elements of $E$ and $E^*$.

Contrarily to {\sc tea{\small 2}}, for implementing {\sc tea{\small 1}}, the extra
element, namely $\boldsymbol{\widehat{\varepsilon}}_{2k}^{(n)}$,  is in the
diagonal above the preceding one. Thus, in addition to the full preceding diagonal, we need to
have also stored the even elements of the previous one.
Thus, in total, one and a half diagonal (of elements of $E$ and $E^*$) has to be stored.

For the simplified topological
$\varepsilon$--algorithms, only elements with an even lower index are used and computed, and
the rules are showed in Table \ref{STEA12}.
Similar considerations as above can be made for the new algorithms. Thus,
for {\sc stea{\small 1}}, by the ascending diagonal technique,
one has only to store the elements of $E$ , that is those with an even lower index,
located in the two preceding diagonals, and
 for {\sc stea{\small 2}}, only the elements of $E$ located in the preceding diagonal.

Of course,  we also have to compute,  by the ascending diagonal technique of Wynn, the elements
of the scalar $\varepsilon$--array, but this is cheep in term of storage requirements and arithmetical operations.
Each new scalar term needed is computed by taking the duality product of the new element of the sequence
with $\mathbf y$.

The storage requirements of the four topological algorithms for computing $\boldsymbol{\widehat{\varepsilon}}_{2k}^{(0)}$ or $\boldsymbol{\widetilde{\varepsilon}}_{2k}^{(0)}$, including the
temporary auxiliary elements, are given in Table \ref{tabb1}.

\begin{table}[!h]
\begin{center}
\begin{tabular}{|c||c|c||c|}
\hline
{\small algorithm} &  {\small  \# elements ${\boldsymbol \varepsilon}$-array} & {\small  in spaces} &  {\small  \# auxiliary elements}\\
\hline \hline
{\small \sc    tea1} & $ 3k$ & $ E, E^*$ & $ 3$\\
{\small \sc   stea1} & $ 2k$ & $ E$ & $ 2$\\
\hline \hline
{\small \sc    tea2} & $ 2k$ &  $ E, E^*$ & $ 3$\\
{\small \sc   stea2} & $ k$ & $ E$ & $ 2$\\
\hline
\end{tabular}
\caption{Storage requirements.}
\label{tabb1}
\end{center}
\end{table}

In the algorithms {\sc stea{\small 1}} and  {\sc stea{\small 2}}, the linear functional
$\mathbf y$ has only to be applied to each new element ${\mathbf S}_n$, and this duality product is immediately
used in the scalar
$\varepsilon$--algorithm for building the new ascending diagonal of the scalar $\varepsilon$--array and, then,
the new ascending diagonal of the topological $\varepsilon$--array consisting only of elements of $E$.

Let us now discuss some possible choices of the  linear functional $\mathbf y \in E^*$.
When $E=\mathbb C^m$, we can define it as $\mathbf y: \mathbf S_n \in \mathbb C^m \longmapsto
<\mathbf y,\mathbf S_n>=(\mathbf y,\mathbf S_n)$, the usual inner product of the
vectors $\mathbf y$ and $\mathbf S_n$, or, more generally,
as $(\mathbf y,\mathbf M\mathbf S_n)$, where $\mathbf M$ is some matrix.

When $E=\mathbb C^{m \times m}$, the linear functional $\mathbf y \in E^*$ can be defined as $\mathbf y: \mathbf S_n \in \mathbb C^{m \times m} \longmapsto
<\mathbf y,\mathbf S_n>=\mbox{trace}(\mathbf S_n)$ or, more generally, for $ \mathbf S_n \in \mathbb C^{m \times s}$, as $\mbox{trace}(\mathbf Y^T \mathbf S_n)$
where $\mathbf Y \in \mathbb C^{m \times s}$.
We can also define $\mathbf y$ by $(\mathbf u,\mathbf S_n \mathbf v)$, where $\mathbf u \in \mathbb C^m$ and
$\mathbf v \in \mathbb C^s$.

The simplified topological $\varepsilon$--algorithms for implementing the topological Shanks transformations also allow us to improve its numerical stability.
This is due to the existence, for the scalar $\varepsilon$--algorithm,
of particular rules derived by Wynn \cite{sr} for this purpose,
while the topological $\varepsilon$--algorithms have no particular rules (such rules also exist for other acceleration algorithms \cite{tmrz}).
These particular rules are used as soon as, for some fixed $p$,
$$|\varepsilon_k^{(n+1)}-\varepsilon_k^{(n)}|/|\varepsilon_k^{(n)}| < 10^{-p}.$$
When this condition is satisfied, a so--called {\it singularity} occurs, the particular rule is used instead of the normal one
for computing a term of the table that will be affected by numerical instability,
and a counter $\sigma$ indicating the number of singularities detected is increased by 1.

\begin{table}[here]
\begin{center}
\begin{tabular}{|c||c|c||c|c|}
\multicolumn{1}{c}{~}& \multicolumn{2}{c}{$p=12$}&\multicolumn{2}{c}{$p=13$}\\
\hline
algorithm &  $ \sigma$ & ${  \|{\boldsymbol\varepsilon}_{10}^{(0)}\|_{\infty}}$ & $ \sigma$ &
${ \|{\boldsymbol\varepsilon}_{10}^{(0)}\|_{\infty}}$ \\
\hline \hline
{\small \sc  tea1\phantom{s--x}} &   &    $9.92 \times 10^{-1}$ &&  $9.92\times 10^{-1}$\\
\hline
{\small \sc  stea1--x} &  $2$ &  $9.35 \times 10^{-4}$ & $0$& $1.15$ \\
\hline \hline
{\small \sc   tea2\phantom{s--x}} &  &  $6.01$ &&   $6.01$\\
\hline
{\small \sc  stea2--1} &  $2$ &  $9.42 \times 10^{-13}$  & $0$& $3.10$\\
{\small \sc  stea2--2} &   $2$&  $1.67 \times 10^{-12}$  & $0$& $3.10$\\
{\small \sc  stea2--3} &  $2$ &  $9.46 \times 10^{-13} $ & $0$& $3.10$\\
{\small \sc  stea2--4} &   $2$&  $1.66 \times 10^{-12}$  & $0$& $3.10$\\
\hline
\end{tabular}
\caption{Performances of the various algorithms for sequence of vectors.}
\label{tabv}
\end{center}
\end{table}

Let us give two numerical examples which show the gain
in numerical stability and accuracy  brought by the simplified topological
$\varepsilon$--algorithms.
We denote by {\small \sc stea1--1}, {\small \sc    stea1--2}, {\small \sc    stea1--3}
and {\small \sc    stea1--4}, respectively, the four rules (\ref{final})--(\ref{final3}) of the
first simplified topological $\varepsilon$--algorithm, and a similar notation for
the equivalent forms  of the second simplified topological $\varepsilon$--algorithm
given in Section \ref{STEAsection}.

We consider the sequence of vectors defined by
$$\begin{array}{@{}l@{~}llll@{}}
\mathbf S_0= \mathbf r, &
\mathbf S_1= (1,\ldots,1)^T, &
\mathbf S_2= (1,\ldots,1)^T+ 10^{-11} \mathbf r, &
\mathbf S_3= \mathbf S_0, &
\mathbf S_4 = \mathbf S_0+ 10^{-11} \mathbf r\\
\multicolumn{5}{@{}l@{}}{\mathbf S_n = 3\mathbf S_{n-1} -\mathbf S_{n-2}+2\mathbf
S_{n-3}+\mathbf S_{n-4}-5\mathbf S_{n-5}, \quad n=5,6,\ldots,}
\end{array}
$$
where $\mathbf r$ is a fixed random vector  whose components are uniformly distributed in $[0,1]$.
Since the vectors $\mathbf S_n$ satisfy a linear difference equation of order 5, we must have
$\boldsymbol{\widehat{\varepsilon}}_{10}^{(n)}= \boldsymbol{\widetilde{\varepsilon}}_{10}^{(n)} =\mathbf 0$
for all $n$. With vectors of dimension 10000 and $\mathbf y=(1,\ldots,1)^T$,
we obtain the results of Table \ref{tabv} ({\small \sc  stea1--x} denotes any of the forms of the algorithm). The notation ${\boldsymbol\varepsilon}_{10}^{(0)}$ corresponds to
${\boldsymbol{\widehat{\varepsilon}}_{10}^{(0)}}$ for {\small\sc tea1} and {\small\sc stea1--x}, and to
${\boldsymbol{\widetilde{\varepsilon}}_{10}^{(0)}}$ for {\small\sc tea2} and {\small\sc stea2--x}.
We can see the important improvement obtained by {\small \sc  stea2} when
$\sigma=2$ singularities have been detected and treated.

\begin{table}[hbt]
\begin{center}
\begin{tabular}{|c||c|c||c|c||c|c|}
\multicolumn{1}{c}{~}& \multicolumn{2}{c}{$p=7$}&\multicolumn{2}{c}{$p=8$}&\multicolumn{2}{c}{$p=10$}\\
\hline
algorithm &  $ \sigma$ & $ \|{\boldsymbol\varepsilon}_{10}^{(0)}\|_{\infty}$ & $ \sigma$ & $ \|
{\boldsymbol\varepsilon}_{10}^{(0)}\|_{\infty}$ &  $ \sigma$ & $ \|{\boldsymbol\varepsilon}_{10}^{(0)}\|_{\infty}$ \\
\hline \hline
{\small \sc   stea1{--x}} &  2 &  $6.14 \times 10^{-7}$ & 1& $1.04$& 0& $1.04$\\
\hline
\hline
{\small \sc    stea2--1} &  2 &  $1.31 \times 10^{-12}$  & 1& $3.02$& 0& $3.02$\\
{\small \sc    stea2--2} &  2 &  $1.05 \times 10^{-12}$  & 1& $3.02$& 0& $3.02$\\
{\small \sc    stea2--3} &   2&  $1.30 \times 10^{-12}$  & 1& $3.02$& 0& $3.02$\\
{\small \sc    stea2--4} &  2 &  $1.04 \times 10^{-12}$  & 1& $3.02$& 0& $3.02$\\
\hline
\end{tabular}
\caption{Performances of the various algorithms for sequence of matrices.}
\label{tabm}
\end{center}
\end{table}
\indent Let us now consider the case where the $\mathbf S_n$ are $m \times m$ matrices constructed
exactly by the same recurrence relation as the vectors of the previous example and with the
same initializations ($\mathbf r$ is now a fixed random matrix). The linear functional
$\mathbf y$ is defined as
$<\mathbf y, \mathbf S_n>=\mbox{trace}(\mathbf S_n)$. For $m=2000$, we obtain the results of
Table \ref{tabm}. We remark that the treatment of only one singularity is not enough for avoiding
the numerical instability, but that, when $\sigma=2$, both
{\small \sc  stea1} and {\small \sc  stea2} give pretty good results, which underlines
the importance
of detecting the singularities and using the particular rules.

To end of this section, we want to add two remarks. In all our tests,
the results obtained with  {\sc stea{\small 2}} are often better. This could be due to the
fact that {\sc stea{\small 2}}
computes a combination (with the same coefficients) of
$\mathbf S_{n+k},\ldots,\mathbf S_{n+2k}$ instead of  $\mathbf S_{n},\ldots,\mathbf S_{n+k}$ for
{\sc stea{\small 1}}. Also, from the tests performed, we cannot
decide which of the four equivalent forms of each simplified algorithm seems to be the most stable one.
However, these experiments suggest us to avoid  using terms of the scalar $\varepsilon$--algorithm
with an odd lower index, and thus, in the next section,
where some applications of the topological Shanks transformations are presented,
we will only show the results obtained by
{\small \sc    stea1--3} and {\small \sc    stea2--3}.

\section{Applications}
\label{appli}

We will consider two types on applications, one with sequences of vectors and another one with sequences of matrices.

\subsection{System of equations}
\label{linit}

There exist many iterative methods for solving systems of nonlinear and linear equations.
For nonlinear systems, the $\varepsilon$--algorithms (scalar, vector, topological) lead to methods with a quadratic convergence under some assumptions.
For linear systems, these algorithms are, in fact, direct methods but they cannot be used in practice since the computation of the exact solution requires too much storage. However, they can be used for accelerating the convergence of iterative methods. Let us look at such an example.

Kaczmarz method \cite{kacz} is an iterative method for linear equations.
It is, in particular, used in  tomographic imaging where it is known as the {\it Algebraic Reconstruction Technique}
({\sc art}). It is well suited for parallel computations and large--scale problems because each step requires only
one row of the matrix (or several rows simultaneously in its block version) and no matrix-–vector products are needed.
It is always converging but its convergence is often quite slow. Procedures for its acceleration were studied in \cite{akrk}.
For the {\tt parter} matrix of dimension 5000 of the {\sc matlab}$^\copyright$ matrix toolbox with $\mathbf x=(1,\ldots,1)^T$,
$\mathbf b=\mathbf A\mathbf x$ computed accordingly, and $\mathbf y=\mathbf b$, Kaczmarz method achieves an error of
$3.44 \times 10^{-1}$ after 48 iterations. With $k=1,3$ or $5$, all our algorithms produce an error between $10^{-12}$ and $10^{-13}$ after 41, 24, and 19 iterations respectively.

Kaczmarz method belongs to the class of {\it Alternating Projection Methods} for finding a point in the intersection of several subsets of an Hilbert space. They are studied in \cite{escal} where some  procedures for their acceleration are given. The algorithms proposed in this paper could be helpful in this context.

\subsection{Matrix equations}

The solution of matrix equations and the computation of matrix functions have important applications \cite{from}. For example, the construction of iterative preconditioners, Newton's method for computing the inverse of a matrix which often converges slowly at the beginning before becoming quadratic, iterative methods for its square root, or for solving the algebraic Riccati equation, or the Lyapunov equation, or the Sylvester equation, or computing the matrix sign function, or decomposition methods for the computation of eigenvalues (method of Jacobi, algorithms LR and QR, etc.).
Let us now consider examples where a sequence of square matrices $\mathbf S_n \in \mathbb R^{m \times m}$ has to be accelerated.

\vskip 2mm

{\bf Example 1:} A new inversion--free iterative method for obtaining the minimal Hermitian positive definite solution
of the matrix rational equation $F(\mathbf S)=\mathbf S + \mathbf A^*\mathbf S^{-1}\mathbf A-\mathbf I=\mathbf 0$, where $\mathbf I$ is the identity matrix and $\mathbf A$ is a given nonsingular matrix was proposed in \cite{marcos}.
It consisted in the following iterations, denoted {\sc ns},
$$\mathbf S_{n+1}=2\mathbf S_n-\mathbf S_n \mathbf A^{-*}(\mathbf I-\mathbf S_n)\mathbf A^{-1}\mathbf S_n,\quad n=0,1,\ldots,$$
with $\mathbf S_0=\mathbf A \mathbf A^*$. All iterates $\mathbf S_n$ are Hermitian. The inverse of $\mathbf A$ has to be computed once at the beginning of the iterations, each of them needing three matrix--matrix products since $\mathbf S_n \mathbf A^{-*}$ is the conjugate transpose of $\mathbf A^{-1}\mathbf S_n$. However, in practice, rounding errors destroys the Hermitian character of $\mathbf S_n$ when $n$ grows. Thus, it is better to program the method as written above without using the conjugate transpose of $\mathbf A^{-1}\mathbf S_n$ and, thus, each iteration requires four matrix--matrix products.

We tried the 5 examples given in \cite{marcos} with $\mathbf y$ corresponding to the trace.
For example 1, the gain is only 1 or 2 iterations.
For example 2, the same precision is obtained by {\sc stea{\small 1}} and {\sc stea{\small 2}}
with $k=1$ in 19--20 iterations instead of 27 for {\sc ns}, and in 16 for $k=2$ and $k=3$.
For the examples 3 and 5, no acceleration is obtained since the {\sc ns}
iterations converge pretty well.   For example 4, we have the results of Table \ref{tabb3}
obtained by stopping the iterations when the Frobenius norm of $F(\mathbf S_n)$ becomes smaller
than $2 \times 10^{-15}$ for one of the methods.

\begin{table}[!h]
\begin{center}
\begin{tabular}{|c||c@{~}c|c@{~}l|c@{~}l|}
\multicolumn{1}{c}{~}
& \multicolumn{2}{c}{$k=1$} & \multicolumn{2}{c}{$k=2$} & \multicolumn{2}{c}{$k=3$}\\
\hline
algorithm & \# iter. & Frob. norm & \# iter. & Frob. norm & \# iter. & Frob. norm\\
\hline
\hline
{\sc ns} & 118 & $1.98\times 10^{-15}$&60 &$6.80\times 10^{-9}$ &49 &$1.18\times 10^{-7}$\\
{\sc stea{\small 1}} & 60 & $1.47\times 10^{-15}$ &60 & $1.52\times 10^{-15}$&43 &$1.63\times 10^{-15}$\\
{\sc stea{\small 2}} & 61 & $1.84\times 10^{-15}$&61 & $1.93\times 10^{-15}$&49 &$1.61\times 10^{-15}$\\
\hline
\end{tabular}
\caption{Performances of the various algorithms for the equation $F(\mathbf S)=\mathbf 0$.}
\label{tabb3}
\end{center}
\end{table}

\vskip 2mm

{\bf Example 2:}
More generally, consider now the matrix equation $\mathbf S+\mathbf A^*\mathbf S^{-q}\mathbf A=\mathbf Q$ for  $0<q \leq 1$.
It can be solved by the iterative method (2.6) of \cite{yin}
$$
\begin{array}{lcl}
\mathbf S_n&=&\mathbf Q-\mathbf A^* \mathbf S_n^q \mathbf A\\
\mathbf Y_{n+1}&=&2\mathbf Y_n-\mathbf Y_n \mathbf S_n \mathbf Y_n,
\end{array}
$$
with $\mathbf Y_0= (\gamma \mathbf Q)^{-1}$ and $\gamma$ conveniently chosen.
For the example 4.2 of dimension 5 with $q=0.7$ and $\gamma=0.9985$ treated in \cite{yin},
the Euclidean norm of the error is $1.30\times 10^{-10}$ at iteration 13 of the iterative method,
 while, with $k=3$, it is $5.8 \times 10^{-11}$ for {\sc stea{\small 1}} and
$6.5\times 10^{-14}$ for {\sc stea{\small 2}}.
At iteration 17, the norm of the error of the iterative method is
$4.11\times 10^{-14}$, that of {\sc stea{\small 1}}  is
$3.5\times 10^{-12}$, and {\sc stea{\small 2}} gives $1.49\times 10^{-14}$.

\vskip 2mm

{\bf Example 3:} We consider now the symmetric Stein matrix equation, also called the
discrete--time Lyapunov equation, $\mathbf S-\mathbf A\mathbf S\mathbf A^T=\mathbf F\mathbf F^T$, $\mathbf F \in \mathbb R^{m \times s}$, $s << m$, and where the eigenvalues of $\mathbf A$ are inside the unit disk.
As explained in \cite{jsnuma}, this equation can be solved by the iterative method given in \cite{smith},
$\mathbf S_{n+1}=\mathbf F\mathbf F^T+\mathbf A\mathbf S_n\mathbf A^T$, $n=0,1,\ldots,$ with $\mathbf S_0=\mathbf 0$.
This method converges slowly.
We performed several tests with matrices from the {\sc matlab}$^\copyright$ matrix toolbox. The gain was always around two digits.

\section{Conclusions}

For concluding, let us summarize the characteristics of the old and new algorithms for implementing the topological Shanks transformations, and compare them.

\vspace{0.3cm}
\begin{center}
\noindent
{\small
\begin{tabular}{|@{~}l@{~}|@{~}l@{~}|}
\multicolumn{1}{@{~}c}{\sc Topological $\varepsilon$--alg. ({\sc tea{\small 1}}, {\sc tea{\small 2}})}
& \multicolumn{1}{@{~}c@{~}}{ \sc Simplified   $\varepsilon$--alg. ({\sc stea{\small 1}},
{\sc stea{\small 2}})} \\
 \hline
 --~2 rules; & --~only 1 rule; \\
\parbox[t]{6.2cm}{--~storage of one and a half ascending diagonals for  {\sc tea{\small 1}}, and
one  ascending diagonal for {\sc tea{\small 2}}; } &
\parbox[t]{6.2cm}{
--~storage of two half ascending diagonals for
{\sc stea{\small 1}}, and half of an ascending diagonal for {\sc stea{\small 2};}
}\\
 \parbox[t]{6.2cm}{--~storage of elements of $E$ and  $E^*$;} &
\parbox[t]{6.2cm}{--~storage of only elements of $E$;} \\
\parbox[t]{6.2cm}{--~use of the duality product by elements of $E^*$
and  by $\bf y$ into the rule of the algorithms;} &
\parbox[t]{6.2cm}{--~application of  $\bf y$ only to ${\bf S}_n \in E$, and
no use of the duality product into the rules; }\\
\parbox[t]{6.2cm}{--~convergence and acceleration results difficult to obtain (rules too complicated);}
&\parbox[t]{6.2cm}{--~possibility to prove convergence and acceleration results;}\\
\parbox[t]{6.2cm}{--~numerical instability can be present.} &
\parbox[t]{6.2cm}{--~possibility of improving the numerical stability.} \\
\hline
\end{tabular}
}
\end{center}
\vskip 3mm

%\noindent {\bf Acknowledgments:} We would like to thank Marcos Raydan for exchanges about the iterative methods for %solving the matrix rational equations.

%The work of C.B. was partially supported by the Labex CEMPI (ANR--11--LABX--0007--01).

%%\newpage

\end{document}